\documentclass[11pt]{amsart}
\usepackage{epsfig,color}
\usepackage{xy}
\xyoption{matrix}
\xyoption{arrow}

\newtheorem{lemma}{Lemma}[section]
\newtheorem{theorem}[lemma]{Theorem}
\newtheorem{corollary}[lemma]{Corollary}

\newtheorem{proposition}[lemma]{Proposition}
\newcommand{\rev}{{\operatorname{rev}}}

\newcommand{\F}{\mathcal F}

\renewcommand{\th}{\operatorname{th}}
\newcommand{\pf}{\mathcal F}

\newcommand{\Hom}{\operatorname{Hom}}
\newcommand{\Homneg}{\operatorname{Hom}_{\leq 0}}
\newcommand{\Ext}{\operatorname{Ext}}
\newcommand{\End}{\operatorname{End}}

\renewcommand{\ker}{\operatorname{ker}}
\newcommand{\coker}{\operatorname{coker}}
\newcommand{\add}{\operatorname{add}}
\newcommand{\module}{\operatorname{mod}}

\newcommand{\pt}{\mathcal T}
\newcommand{\T}{\mathcal T}
\newcommand{\U}{\mathcal U}
\newcommand{\C}{\mathcal C}
\newcommand{\D}{\mathcal D}

\newcommand{\Y}{\mathcal Y}
\newcommand{\E}{\mathcal E}
\newcommand{\isom}{\simeq}
\newcommand{\dmone}{\D^{\geq 1}_{\leq m}}
\newcommand{\dmoneplus}{\D^{(\geq 1)+}_{\leq m}}
\newcommand{\dmzero}{\D^{\geq 0}_{\leq m}}
\newcommand{\dmzerominus}{\D^{(\geq 0)-}_{\leq m}}
\newcommand{\dzero}{\D^{\geq 0}}
\newcommand{\dm}{\D_{\leq m}}
\newcommand{\donezerominus}{\D^{(\geq 0)-}_{\leq 1}}

\renewcommand{\mod}{\operatorname{mod}}
 
\newcommand{\stab}{\operatorname{Stab}}

\renewcommand{\mod}{\operatorname{mod}}
\newcommand{\Sub}{\operatorname{Sub}}
\newcommand{\Fac}{\operatorname{Fac}}

\begin{document}

\title[Exceptional sequences]{From $m$-clusters to 
$m$-noncrossing partitions via exceptional sequences}
\author[Buan]{Aslak Bakke Buan}
\address{Institutt for matematiske fag\\
Norges teknisk-naturvitenskapelige universitet\\
N-7491 Trondheim\\
Norway}
\email{aslakb@math.ntnu.no}

\author[Reiten]{Idun Reiten}
\address{Institutt for matematiske fag\\
Norges teknisk-naturvitenskapelige universitet\\
N-7491 Trondheim\\
Norway}
\email{idunr@math.ntnu.no}

\author[Thomas]{Hugh Thomas}
\address{Department of Mathematics and Statistics\\University of 
New Brunswick\\Fredericton NB\\E3B 1J4 Canada
}
\email{hthomas@unb.ca}
\thanks{All three authors were supported by STOR-FORSK grant 167130
  from NFR.  A.B.B. and I.R. were supported by grant 196600 from NFR.  H.T. was supported by an NSERC Discovery Grant.}

%\date{\today}

\begin{abstract} Let $W$ be a finite crystallographic reflection group.  The
generalized Catalan number of $W$ coincides both with the number of 
clusters in the cluster algebra associated to $W$, and with the number
of noncrossing partitions for $W$.  Natural bijections between these two
sets are known.  For any positive integer $m$, both $m$-clusters and 
$m$-noncrossing partitions have been defined, and the cardinality
of both these sets is the Fuss-Catalan number $C_m(W)$.  We give a natural
bijection between these two sets by first establishing a bijection
between two particular sets of exceptional sequences in the 
bounded derived category $D^b(H)$ for any finite-dimensional hereditary
algebra $H$.\end{abstract}


\maketitle

\section*{Introduction}

This paper is motivated by the following problem in combinatorics.  
Let $W$ be a finite crystallographic reflection group.  Associated to $W$ is a positive integer called the generalized
Catalan number, which on the one hand equals the number of clusters in 
the associated cluster algebra \cite{fz}, and on the other hand equals the
number of noncrossing partitions for $W$, see \cite{bessis}.  Natural bijections
between the sets of clusters and noncrossing partitions associated with
$W$ have been found in \cite{read,abmw}.  More generally, for any integer
$m\geq 1$, there is associated with $W$ a set of $m$-clusters introduced
in \cite{fr} and a set of $m$-noncrossing partitions defined in \cite{arm}.  Each
of these sets has cardinality the Fuss-Catalan number $C_m(W)$, see
\cite{fr,arm}.  
The formula for $C_m(W)$ is as follows:

$$C_m(W)=\frac{\prod_{i=1}^n mh+e_i+1}{\prod_{i=1}^n e_i+1},$$
where $n$ is the rank of $W$, $h$ is its Coxeter number, and $e_1,\dots,e_n$ are its
exponents.  

One of our main results is to establish a natural bijection between the 
$m$-clusters and the $m$-noncrossing partitions for any $m\geq 1$.  We 
accomplish this by first solving a related, more general problem about
bijections between classes of exceptional sequences in bounded derived 
categories of finite dimensional hereditary algebras, which is also of
independent interest.  

Let $H$ be a connected hereditary artin algebra.
Then $H$ is a finite dimensional algebra over its centre, which is known
to be a field $k$.   
Examples of such algebras are path algebras over a field of finite quivers with no oriented cycles.
Let $\module H$ be the category of finite
dimensional left $H$-modules and let $\D = D^b(H)$ be the bounded derived category.
An $H$-module $M$ is called {\em rigid} if  $\Ext^1(M,M) = 0$, and 
an indecomposable rigid $H$-module is called {\em exceptional}. The set of
isomorphism classes of exceptional modules is countable, and it has interesting combinatorial structures,
which have been much studied
in the representation theory of algebras,
and in various combinatorial applications of this theory.

We study exceptional objects and sequences in the derived category $\D$.  
With a slight modification of the definition in \cite{kv}, we say that
an object $T$
in $\D$ is {\it silting} if $\Ext^i(T,T)=0$ for $i>0$ and $T$ is
maximal with respect to this property.  
We say that a basic object
$X=X_1\oplus \dots \oplus X_n$ in $\D$ is a $\Homneg$-configuration if all
$X_i$ are exceptional, $\Hom(X_i,X_j)=0$ for $i\ne j$, $\Ext^t(X,X)=0$ for
$t<0$, and there is no subset $\{Y_1,\dots,Y_r\}$ of the indecomposable
summands of $X$ such that $\Ext^1(Y_i,Y_{i+1})\ne 0$ for $ 1 \leq i
< r$ and $\Ext^1(Y_r,Y_1)
\ne 0$.  (Here $n$ denotes the number of isomorphism classes of simple 
$H$-modules.)
It follows from the definition of $\Homneg$-configuration that 
 $\{X_1,\dots,X_n\}$ can be ordered into a complete
exceptional sequence.  For any $m\geq 1$, we say that $X$ is an 
$m$-$\Homneg$-configuration if the $X_i$ lie in $\module H[t]$ for $0\leq t\leq m$.

Given the representation-theoretic interpretation of noncrossing
partitions
provided by \cite{it}, it was reasonable to expect a
representation-theoretic
manifestation of $m$-noncrossing partitions. One approach to developing
such a definition would have been to follow \cite{it} closely, and consider
sequences of finitely generated exact abelian, extension closed
subcategories
with suitable orthogonality conditions.  $\Homneg$-configurations
seemed to provide a more convenient viewpoint. When we have a Dynkin
quiver, the vanishing of $\Hom$ and of $\Ext$, which can be reduced to
$\Hom$, is easy to compute on the AR-quiver.
Hence it is not hard to compute $\Homneg$-configurations in this case.

Our main result is to obtain a natural bijection between silting objects and $\Homneg$-configurations via a certain sequence of mutations of exceptional sequences.  This induces
a bijection between $m$-cluster tilting objects and $m$-$\Homneg$-configurations,
for any $H$.  We also give a bijection between $m$-$\Homneg$-configurations
and $m$-noncrossing partitions for arbitrary $H$.  Specializing to $H$ being
of Dynkin type, we get as an application a bijection betwen $m$-clusters and 
$m$-noncrossing partitions.  

The paper is organized as follows. We first review preliminaries concerning exceptional sequences in module categories
as well as in derived categories. In Section 2, we recall 
the definition of silting objects and $m$-cluster tilting objects, 
and define $\Homneg$-configurations and $m$-$\Homneg$-configurations. We also state the precise version of our main result.
In the next section we give some basic results about mutations of exceptional sequences in the derived category. In
Section 4 we show how to construct $\Homneg$-configurations from silting objects. In the next two sections we finish
the proof of our main result. In Section 7 we give 
the combinatorial interpretation of our main result, 
including a version for the 
``positive'' Fuss-Catalan combinatorics.  
In Section 8, we discuss the relationship 
between our $\Homneg$-configurations and 
Riedtmann's combinatorial configurations from her work on 
selfinjective algebras \cite{Rie1,Rie2}. 
In Section 9 we show how the bijection we have constructed interacts with
torsion classes in $\D$.

We remark that the results in Section 8 have also been
obtained by Simoes \cite{sim}, in the Dynkin case, with an 
approach which is different than ours, and 
and independent from it.

\section{Preliminaries on exceptional sequences}

As before, let $H$ be a finite dimensional connected 
hereditary algebra over a field $k$ which is the centre of $H$, 
and let $\module H$ denote 
the category of finite dimensional left $H$-modules. We assume that $H$ has $n$ simple modules up to isomorphism.  In this section we recall some basic
results about exceptional sequences.  

\subsection{Exceptional sequences in the module category}

A sequence of exceptional objects 
$\E = (E_1,\dots,E_r)$ in $\module H$ is called an {\em exceptional sequence}
if $\Hom(E_j,E_i)=0=\Ext^1(E_j,E_i)$ for $j>i$.  

There are right and left mutation 
operations, denoted respectively $\mu_i$ and $\mu_i^{-1}$, which take exceptional sequences to exceptional 
sequences. 
Given an exceptional sequence $\E = (E_1,\dots,E_r)$, right mutation replaces the subsequence $(E_i,E_{i+1})$ by
$(E_{i+1},E_i^{\ast})$, while left mutation 
replaces the subsequence $(E_i,E_{i+1})$ by
$(E_{i+1}^{!}, E_i)$, for some exceptional
  objects $E_{i+1}^{!}$ and $E_i^{\ast}$.

We need the following facts about exceptional sequences in $\module H$. These are proved in 
\cite{c} (if the field $k$ is algebraically closed) and in \cite{r} in general.

\begin{proposition}\label{oldstuff}
Let $\E = (E_1,\dots,E_r)$ in $\module H$ be an exceptional sequence. Then the following hold:
\begin{itemize}
\item[(a)] $r \leq n$
\item[(b)] if $r <n$, then there is an exceptional sequence
$(E_1,\dots,E_r, E_{r+1}, \dots, E_n)$
\item[(c)] if $r= n-1$, then for a fixed index $j \in \{1, \dots n\}$, 
there is a unique indecomposable $M$, such that $$(E_1, \dots, E_{j-1},M, E_j, \dots E_{n-1})$$ is an  
exceptional sequence 
\item[(d)] for any $i \in \{1, \dots, r-1\}$, we have  $\mu_i^{-1} (\mu_i(\E)) = \E = \mu_i (\mu_i^{-1}(\E))$ 
\item[(e)] the set of $\mu_i$ satisfies the braid relations, i.e. $\mu_i \mu_{i+1} \mu_i =  \mu_{i+1} \mu_i \mu_{i+1}$ for
$i \in \{1, \dots, r-2 \}$, and $\mu_i \mu_j = \mu_j \mu_i$ for $|i-j| >1$
\item[(f)] the action of the set of $\mu_i$ on the set of complete exceptional sequences is transitive.
\end{itemize}
\end{proposition}

An exceptional sequence $\E= (E_1,\dots,E_r)$ is called {\em complete} if $r=n$.

\subsection{Exceptional sequences in derived categories}

Let $\D=D^b(H)$ denote the bounded derived category with translation functor $[1]$, the shift functor. 
This is a triangulated category, see \cite{hap} for general properties of such categories.
It is well known that since $H$ is hereditary,
the indecomposable objects of $\D$ are stalk complexes, i.e. they are up to isomorphism
of the form $M[i]$ for some indecomposable $H$-module $M$ and some
integer $i$. If $X =M[i]$ is an indecomposable object in $\D$, we will write $\overline{X} = M$ for the 
corresponding object in $\module H$.  

It is well-known that the derived category $\D$ has almost split triangles \cite{hap}, and hence an AR-translation $\tau$,
or equivalently a Serre-functor $\nu$, where we have $\nu = \tau [1]$.
We have the AR-formula $\Hom_{\D}(X,Y) \simeq D\Hom(Y, \tau X)$, for all objects $X,Y$ in $\D$.

It is convenient to consider also exceptional sequences in
the derived category $\D$. 
Let $\E = (E_1, E_2, \dots, E_r)$ be a sequence of indecomposable objects in $\D$. It is called an {\em exceptional
sequence} in $\D$ if 
$\overline{\E} = (\overline{E_1}, \overline{E_2}, \dots, \overline{E_r})$ 
is an exceptional sequence in $\module H$, and {\em complete} if $\overline{\E}$ is complete.

In Section \ref{S:Mut} we will describe a mutation operation on exceptional sequences in $\D$. For this we need the following
preliminary results.

\begin{lemma}\label{atmostone} Let $(E,F)$ be an exceptional sequence in $\D$. Then $\Ext^i(E,F)= \Hom_{\D}(E,F[i])$ is
nonzero for at most one integer $i$, and $\Ext^i(F,E) = \Hom_{\D}(F,E[i]) = 0$ for
all $i\in \mathbb Z$.  
\end{lemma}

\begin{proof} We provide the short proof from \cite{brt} for the convenience
of the reader.  

It suffices to check the statements for $(\overline{E}, \overline{F})$.
By results from \cite{c,r}, we can consider $(\overline{E}, \overline{F})$ as an exceptional sequence
in a hereditary module category, say $\module H'$, with $H'$ of rank 2, and such that $\module H'$ has a full 
and exact embedding into $\module H$.
For a hereditary algebra $H'$ of rank 2, the only exceptional indecomposable modules are
preprojective or preinjective. Hence, a case analysis of the possible exceptional sequences in
$\module H'$ for such algebras, gives the first statement.
The second statement is immediate from the definition of exceptional sequence.
\end{proof}

There is a general notion of exceptional sequences in triangulated categories, see \cite{bond,gk}.
Note that in our setting, this definition is equivalent to our definition. 
This follows from combining the fact that indecomposables in $\D$ are stalk complexes
with the second part of Lemma \ref{atmostone}.

\section{Silting objects and $\Homneg$-configurations}

In this section
we recall some basic properties of silting objects, and introduce the
notion of $\Homneg$-configurations.

\subsection{Silting objects}

A basic object $Y$ in $\D$ is called a {\em partial silting object} if 
$\Ext^i_{\D}(Y,Y)=0$ for $i\geq 1$, and {\em silting} if it is maximal with respect to this property.
Note that this differs slightly from the original definition in \cite{kv}.
It is known (see \cite{brt}) that a partial silting object $Y$ is silting if and only if it
has $n$ indecomposable direct summands. If a silting object $Y$
  is in $\mod H$, it is called a {\em tilting module}. 

The following connection with exceptional sequences is a special case of \cite[Theorem 2.3]{ast}. 
We include the sketch of a proof for convenience.

\begin{lemma}\label{l:silt_exc}
If $Y$ is partial silting in $\D$, there is a way to order its indecomposable direct summands
to obtain an exceptional sequence $(Y_1,\dots,Y_r)$ in $\D$.  
\end{lemma}

\begin{proof} 
Assume that $A[u]$ and $B[v]$ are indecomposable direct summands of $Y$ where
$A,B$ are $H$-modules. If $v>u$, then $\Ext^i_{\D}(B, A) = 0$ for all $i$, so we 
put $A[u]$ before $B[v]$ in the exceptional sequence.
For a fixed degree $d$, assume there are $t \geq 1$ direct summands of $Y$ of degree $d$. The direct sum of these $t$ summands 
is the shift of a rigid module in $\module H$. By \cite{hr}, there
are no oriented cycles in the quiver of the endomorphism ring of a
rigid module. Hence, 
there is an ordering on these $t$ summands, say $A_1, \dots, A_t$, such that $\Hom(A_j,A_k) = 0$ for $j>k$.  
\end{proof}

An exceptional sequence is called {\em silting} if it is induced by a silting object as in
Lemma \ref{l:silt_exc}.

An object $M$ in $\D$ is called a {\em generator} if $\Hom_{\D}(M,X[i]) = 0$ for all $i$ only if $X=0$. 
For a hereditary algebra, the indecomposable projectives can be ordered to form an exceptional sequence.
Hence, by transitivity of the action of mutation on exceptional sequences (Proposition \ref{oldstuff}) (e)), the direct sum of 
the objects in an exceptional sequence is a 
generator. Thus we obtain the following consequence of Lemma \ref{l:silt_exc}.

\begin{lemma}\label{silting_generates}
Any silting object in $\D$ is a generator.
\end{lemma}

For a positive integer $m$, the $m$-cluster category is the orbit category
$\C_m = \D / \tau^{-1}[m]$, see \cite{bmrrt,keller,t,z,w}. It is canonically triangulated by \cite{keller}.
An object $T$ in $\C_m$ is called {\em maximal rigid} if 
$\Ext_{\C_m}^i(T,T)= 0$ for $i= 1, \dots , m$, and $T$ is maximal with respect to this property.
An object $T$ in $\C_m$ is called {\em $m$-cluster tilting} if for any object $X$, we have that
$X$ is in $\add T$ if and only if $\Ext_{\C_m}^i(T,X)= 0$ for $i= 1, \dots , m$.
It is known by \cite{w,zz} that $T$ is maximal rigid if and only if it is an $m$-cluster tilting object.
Let $\dmoneplus$ denote the full subcategory of $\D$ additively generated by: 
the injectives in $\module H$,
together with $\module H[i]$ for $1\leq i \leq m$. 
Every object in $\C_m$ is induced by an object contained in $\dmoneplus$.

In \cite{brt} it is shown that silting objects contained in $\dmoneplus$ are in 1-1 
correspondence with $m$-cluster tilting objects. 
We consider this an identification, and from now on we will refer to silting objects contained in $\dmoneplus$
as $m$-cluster tilting objects.

\subsection{$\Hom_{\leq 0}$-configurations and $m$-$\Hom_{\leq 0}$-configurations}
In this subsection we introduce new types of objects in $\D$.  They
are related to the combinatorial configurations investigated in
\cite{Rie1,Rie2}, and will turn out to be closely related to 
noncrossing partitions.   

A basic object $X$ in $\D$ is a {\em $\Hom_{\leq 0}$-configuration} if
\begin{enumerate}
\item [(H1)] $X$ is the direct sum of $n$ exceptional indecomposable
summands $X_1,\dots,X_n$, where $n$ is the number of simple modules of $H$.
\item [(H2)] $\Hom(X_i,X_j)=0$ for $i \neq j$.
\item [(H3)] $\Ext^t(X,X)=0$ for $t<0$.  
\item [(H4)] there is no subset $\{Y_1,\dots,Y_r\}$ of the indecomposable 
summands
of $X$ such that $\Ext^1(Y_i,Y_{i+1})\ne 0$ and $\Ext^1(Y_r,Y_1)\ne 0$.  
\end{enumerate}

\begin{lemma} \label{hfour}
The indecomposable summands of a $\Homneg$-configuration can
be ordered into a complete exceptional sequence. 
\end{lemma}

\begin{proof}  Let $X$ be a $\Homneg$-configuration, and let
$X=\bigoplus_i A_i[i]$, with $A_i\in \mod H$, and 
where all but finitely many of the $A_i$ are
zero.   Each $A_i$ is the direct sum of finitely many indecomposables in
$\mod H$ with no morphisms between them, so (H4) suffices to conclude
that they can be ordered into an exceptional sequence.  Now concatenate
the sequences in {\it decreasing} order with respect to $i$.  This 
implies that if $E,F$ are indecomposable summands of $X$ lying in
$\mod H[e]$ and $\mod H[f]$ respectively with $e<f$, 
then $F$ will precede $E$ in
the sequence.  We must therefore show that $\Ext^j(E,F)=0$ for all $j$.  
This is 
true for $j\leq 0$ by (H2) and (H3), and for $j>0$ because $e<f$. 
\end{proof}

Note that it is also the case that, for $X$ in $\D$,  
if the summands of $X$ can be ordered
into a complete exceptional sequence, then (H4) is necessarily satisfied.

If $X$ is a $\Homneg$-configuration, we refer to an exceptional  
sequence on the indecomposable summands of $X$ as a $\Hom_{\leq 0}$-configuration exceptional sequence. 
A $\Hom_{\leq 0}$-configuration will be called an 
{\em $m$-$\Hom_{\leq 0}$-configuration} if it is contained in the full subcategory 
$\dmzero$, whose indecomposables are in $\module H[i]$ for $0\leq i \leq m$.

%Let $\donezerominus$ denote the full subcategory additively generated by 
%$\module H[0]$, except the projectives, and $\module H[1]$.
%A link between
%$\Homneg$-confugations inside $\donezerominus$ and Riedtmann's
%combinatorial configurations has been studied by Simoes \cite{sim}.

\subsection{Main results}

One aim of this paper is to use mutation of exceptional sequences to 
establish the following result. 

\begin{theorem} \label{maintheorem}
There are bijections between
\begin{itemize}
\item[(a)] exceptional sequences which are silting and  
exceptional sequences which are $\Hom_{\leq 0}$-configurations.
\item[(b)] silting objects and $\Hom_{\leq 0}$-configurations.
\item[(c)] $m$-cluster tilting objects and $m$-$\Hom_{\leq 0}$-configurations (for any $m \geq 1)$. 

\end{itemize}
\end{theorem}

We prove (a) in Section \ref{s:silt_hom} and (b) in Section \ref{s:bijection}, while (c) is proved in Section 
\ref{s:cluster_config}.  In Section 7 we apply (c) in finite type to 
obtain a bijection between $m$-noncrossing partitions in the sense of
\cite{arm} and $m$-clusters in the sense of \cite{fr}.

\section{Mutations in the derived category}\label{S:Mut}

In this section we give some basic results on mutations of exceptional sequences in the bounded
derived category. This is the main tool used in the proof of 
Theorem \ref{maintheorem}.
We also compare mutations in $\D$ with mutations in $\module H$.
The results in this section can also be found in e.g. \cite{bond, gk}.
We include proofs, for completeness and for the convenience of the reader. 

We start with the following observation.
\begin{lemma}\label{movetoend} Let $(E_1,\dots,E_n)$ be a complete exceptional
sequence in $\D$.  
For any complete exceptional sequence $(E_2,\dots,E_n,X)$, we have
$X=\nu^{-1} E_1[k]$ for some $k$.  
\end{lemma}

\begin{proof}
Since the sequence $(E_1,\dots,E_n)$ is exceptional, we know that 
$\Ext^i_{\D}(E_j,E_1)=0$ for $j>1$ and all $i$.  By Serre duality for $\D$, this
implies that $\Ext^i_{\D}(\nu^{-1} E_1, E_j)$=0 for all $i$ and all $j>1$.  
From this it follows that $(\overline{E_2},\dots,\overline{E_n}, \nu^{-1} \overline{E_1})$
is exceptional in $\module H$. The claim now follows from
Proposition \ref{oldstuff} (c).
\end{proof}

We now describe mutation of exceptional sequences in $\D$, and show that 
it is compatible with mutation in the module category.

For an object $Y$ in $\D$, we
write $\th(Y)$ for the thick additive full subcategory of $\D$ generated by $Y$.  
Note that if $Y$ is exceptional, the objects of $\th(Y)$ are direct sums
of objects of the form $Y[i]$.

Define an operation $\hat \mu_i$ on exceptional sequences in $\D$
by replacing the  pair $(E_i,E_{i+1})$ by the pair $(E_{i+1},E_i^{\ast})$, where $E_i^{\ast}$ is defined by taking 
$E_i \rightarrow Z$ to be the minimal 
left $\th(E_{i+1})$-approximation of $E_i$, and completing to a triangle:

$$E_i^{\ast} \to E_i \to Z  \to E_i^{\ast} [1]$$

\noindent Note that $Z$ is of the form $E_{i+1}^r[p]$, by Lemma \ref{atmostone}
(in other words, $Z$ is concentrated in one degree).

Similarly, we define $\hat \mu_i^{-1}$ of $(E_1,\dots,E_r)$ by taking 
$Z\rightarrow E_{i+1}$ to be the minimal right 
$\th(E_i)$ approximation of $E_{i+1}$,
completing to a triangle 
$$ Z \rightarrow E_{i+1}\rightarrow  E_{i+1}^{!} \rightarrow Z[1] $$  
and replacing the pair $(E_i,E_{i+1})$ with $(E_{i+1}^{!},E_i)$.

We recall the following well-known properties of exceptional objects
in $\module H$. The proofs of these
are contained in \cite{bong}, \cite{hr} and \cite{rs}, see also \cite{hub}.

\begin{lemma}\label{modulestuff}
Let $E,F$ be exceptional modules, and assume $\Hom(F,E) = 0 = \Ext^1(F,E)$.
\begin{itemize} 
\item[(a)] Let $f \colon E \to F^r$ be a minimal left $\add F$-approximation.
If $\Hom(E,F) \neq 0$, then $f$ is either an epimorphism or a
monomorphism.  
\item[(b)] If $\Ext^1(E,F) \neq 0$, there is an 
extension 
\[
0 \to F^r \to U \to E \to 0 \]
with the property that $\Hom(F^r,F') \to \Ext^1(E,F')$ is a surjection for any $F' \in \add F$.
\end{itemize}
\end{lemma}

\begin{lemma}\label{mutationlemma}
\begin{enumerate}
\item[(a)]
The operations $\hat \mu_i$ and $\hat \mu_i^{-1}$ are mutual inverses.

\item[(b)] Let $\E$ be an exceptional sequence in $\D$, then $\mu_i (\overline{\E})
= \overline{\hat \mu_i (\E)}$ and 
$\mu_i^{-1} (\overline{\E})
= \overline{{\hat \mu_i^{-1}} (\E)}$.

\item[(c)]\label{three}
If $(E_1,\dots,E_n)$ is a complete exceptional sequence in $\D$, then
$$\mu_{n-1}\dots\mu_1(E_1,\dots,E_n)=(E_2,\dots,E_n,\nu^{-1}E_1).$$

\item[(d)] The operators $\hat \mu_i$ and $\hat \mu_j$ satisfy the braid 
relations.  

\item[(e)] Let $(A,B,C)$ be an exceptional sequence in $\D$.  Let $\hat \mu_1 \hat \mu_2 
(A,B,C)=(C,A^{\ast},B^{\ast})$.  Then $\Ext^\bullet(A,B)\isom\Ext^\bullet(A^{\ast},B^{\ast})$.

\end{enumerate}
\end{lemma}

\begin{proof} 
(a) Let $$ E_{i+1}^r [p-1] \to E_i^{\ast} \to E_i \to E_{i+1}^r[p]  $$
be the approximation triangle defining $E_i^{\ast}$. Since $\Hom(E_{i+1}, E_i[j])= 0$ for all $j$ 
by Lemma \ref{atmostone},
it is clear that the map $E_{i+1}^r [p-1] \to E_i^{\ast}$ is a right $\th(E_{i+1})$-approximation of $E_i^{\ast}$.
The assertion follows from this and the dual argument.

For (b), let us recall how right mutation $\mu_i$ is defined in
$\module H$. For $(E,F)$ an exceptional pair in $\module H$ we have that 
$\mu_1(E,F) = (F,E^{\ast})$. Let $f:E\rightarrow F^r$ be the minimal  left $\add
  F$-approximation. Then the module 
$E^{\ast}$ is defined as follows: 

$$E^{\ast} = \begin{cases} 
E & \text{if $\Hom(E,F) = 0 = \Ext^1(E,F)$} \\
\ker f & \text{if $\Hom(E,F) \neq 0$ and $f$ is an epimorphism} \\
\coker f & \text{if $\Hom(E,F) \neq 0$ and $f$ is a monomorphism} \\
U & \text{if $\Ext^1(E,F) \neq 0$ with $U$ as in (\ref{modulestuff} (b) )} 
\end{cases}$$

Note that at most one of $\Hom(E,F)$ and $\Ext^1(E,F)$ is non-zero, by Lemma \ref{atmostone}.
It is now straightforward to check that in all cases we have
$\mu_i (\overline{\E})
= \overline{\hat \mu_i (\E)}$. This proves part (b) of Lemma \ref{mutationlemma}.

For (c) consider the
approximation triangles
$$X_j \overset{f_{j-1}}{\rightarrow} X_{j-1} \rightarrow
E_{j}^{r_{j}}[k_j] \to X_j[1]$$
where we let
$X_i$ be the object in the $i$-th place of 
$\mu_{i-1}\dots \mu_1 (E_1,\dots,E_n)$; that is to say, the object obtained
by $i-1$ successive mutations of $E_1$.  
We want to show that $\Hom(X_n, X_1) \neq 0$.
We have a sequence of morphisms 

$$\xymatrix{ X_n \ar[r]^{f_{n-1}}& X_{n-1} \ar[r]^{f_{n-2}}&\cdots \ar[r]^{f_2}
& X_2 \ar[r]^{f_1}&
X_1}$$
We claim that the composition of the morphisms is nonzero.  
Without loss of generality we can replace $E_j^{r_j}[k_j]$ with $E_j^{r_j}$,
and assume that all approximations are non-zero. 
We first consider the composition 
$f_{1} f_{2}$.

Apply the octahedral axiom:

$$\xymatrix{
& X_1 \ar@{=}[r] & X_1 \\
E_3^{r_3}[-1] \ar[r] & X_3 \ar[r]^{f_2}\ar[u] & X_2 \ar[r]\ar[u]^{f_1} & E_3^{r_3}\\
E_3^{r_3}[-1] \ar[r]\ar@{=}[u] & Y \ar[r]\ar[u] & E_2^{r_2}[-1] \ar[r]\ar[u] &
E_3^{r_3}\ar@{=}[u]\\
& X_1[-1] \ar[u]\ar@{=}[r] & X_1[-1]\ar[u] }$$

If the composition $f_1f_2$ is zero then the left column splits, and 
$Y\isom X_3\oplus X_1[-1]$.  
By Lemma \ref{atmostone}, since 
$(X_1,E_3)$ is exceptional, then $\Hom(E_3,X_1)=0$.  

Thus we have a pair of triangles and a commutative diagram, where the
second vertical arrow is projection onto the second summand

$$ \xymatrix{
E_3^{r_3}[-1]\ar[d]\ar[r]& X_3\oplus X_1[-1]\ar[r]\ar[d]&  E_2^{r_2}[-1] \ar[r]\ar@{.>}[d]
 & E_3^{r_3} \ar[d]\\
0\ar[r]&X_1[-1]\ar[r]&X_1[-1]\ar[r]&0}$$
which implies the existence of the dotted arrow.  This forces the right
column in the previous diagram to split, which is a contradiction.  Hence $f_1 f_2 \neq 0$.

The same argument can be iterated, taking the left column from the previous
diagram and using it as the right column for another octahedron.  One then
uses the fact that $\Hom(E_4,X_1)=0$ in a similar way to the above, and obtains $(f_1 f_2) f_3 \neq 0$.
By further iterations one obtains $f_1 f_2 \dots f_{n-1} \neq 0$.  

By Lemma \ref{movetoend},
we know that $X_n=\nu^{-1}X_1[j]$ for some $j$, so the fact that there
is a nonzero morphism from $X_n$ to $X_1$ implies that $X_n=\nu^{-1} X_1$
as desired. This completes the proof of (c).

The nontrivial case of (d) is to show

$$ (\hat\mu_1\hat\mu_{2}\hat\mu_1)(X,Y,Z) = 
(\hat\mu_{2}\hat\mu_1\hat\mu_{2})(X,Y,Z) $$
(and similarly for left mutation).  
The first terms of the sequences on the left and right hand sides are 
both $Z$, and the second terms agree by definition.  The third terms
agree by part (c), after passing to the derived category of 
the rank 3 abelian category containing $\bar X, \bar Y, \bar Z$.  

For (e) consider the exchange triangles
$$B^{\ast} \to B \to C^u \to B^{\ast}[1]$$ and
$$A^{\ast} \to A \to C^v \to A^{\ast}[1]$$

Applying $\Hom(A^{\ast},\ )$ to the first and $\Hom(\ , B)$ to the second triangle
one obtains
the long exact sequences
$$\Ext^{i-1}(A^{\ast}, C^u) \to \Ext^i(A^{\ast},B^{\ast}) \to \Ext^i(A^{\ast},B) \to \Ext^i(A^{\ast},C^u)$$
and
$$\Ext^i(C^v,B) \to \Ext^i(A,B) \to \Ext^i(A^{\ast},B) \to \Ext^{i+1}(C^v,B)$$
The first and last term of both sequences vanish. Hence we obtain the isomorphisms
$\Ext^i(A,B) \simeq \Ext^i(A^{\ast},B) \simeq \Ext^i(A^{\ast},B^{\ast})$ for each $i$.
\end{proof}

From now on, we shall omit the carets from $\hat\mu_i$, $\hat\mu_i^{-1}$.  

\section{From silting objects to $\Hom_{\leq 0}$-configurations via exceptional sequences}\label{s:silt_hom}

In this section we consider exceptional sequences induced by 
silting objects and by $\Hom_{\leq 0}$-configurations in $\D$.
Recall that an exceptional sequence $\Y = (Y_1, Y_2, \dots, Y_n)$ is called silting if
$Y_1 \oplus Y_2 \oplus \dots \oplus Y_n$ is a silting object.
Note that different exceptional sequences can in this way give rise to the same silting object, and
recall that by Lemma \ref{l:silt_exc} any silting object can be obtained from
an exceptional sequence in this way.

Recall also that any $\Hom_{\leq 0}$-configuration gives rise to a (not necessarily unique) exceptional sequence,
and that such exceptional sequences are called $\Hom_{\leq 0}$-configuration
exceptional sequences.

We will prove part (a) of Theorem \ref{maintheorem}: that silting exceptional sequences are in 1--1 correspondence with $\Hom_{\leq 0}$-configuration 
exceptional sequences.
This will be proved by 
considering the following product of mutations:

\begin{equation}\label{order}
\mu^{(n)}_{\rev}=\mu_{n-1}(\mu_{n-2}\mu_{n-1})\dots (\mu_1\dots \mu_{n-1}).  
\end{equation} 
where we sometimes omit the superscript $(n)$ from $\mu^{(n)}_{\rev}$. 
The same sequence of mutations has been considered in \cite{bond} in
a related context.  

Using that the $\mu_i$ satisfy the braid relations, $\mu_{\rev}$ can be expressed
in various ways, in particular as
\begin{equation}\label{altorder}
\mu_{\rev}=\mu_1(\mu_2\mu_1)(\mu_3\mu_2\mu_1)\dots(\mu_{n-1}\dots \mu_1).
\end{equation}

We say that a mutation $\mu_i$ of an exceptional sequence 
$Y=(Y_1,\dots,Y_n)$ is {\em negative} if the left approximation
is of the form:

$$ Y_{i+1}^r[j] \to Y_i^{\ast}\to  Y_i \to 
Y_{i+1}^r[j+1] $$
where $j$ is negative and
{\em non-negative} if $j\geq 0$.  

Similarly, we say that $\mu_i^{-1}$ is {\em negative} if $j$ is negative in the approximation

$$Y_{i-1}^r[j] \rightarrow Y_i \rightarrow Y_i^{\ast} \rightarrow 
Y_{i-1}^r[j+1]$$
and {\em non-negative} if $j\geq 0$.
  
It is immediate from the definitions that if $\mu_i$ is negative, 
then $\mu_i^{-1}$ applied to $\mu_i(Y)$ will also be negative.  

\begin{lemma} \label{goodmut}
Assume that the exceptional sequence $(Y_1,\dots,Y_n)$ is silting.  
Consider the process of applying $\mu_{\rev}$ to it in the order given by (\ref{order}).  
Then each mutation will be negative.
\end{lemma}

\begin{proof} 
Note that $\mu_1 \dots \mu_{n-1}(Y_1,\dots,Y_n) = (Y_n, Y_1^{\ast}, \dots, Y_{n-1}^{\ast})$.   
Since $Y$ is a silting object, each of the mutations
$Y_i\rightarrow Y_n^r[j]$ is negative. The claim can now be proved by induction,
after using Lemma \ref{mutationlemma} (e), which
guarantees that $(Y_1^{\ast}, \dots, Y_{n-1}^{\ast})$ form a silting object in the subcategory
of $\D$ which they generate. 
\end{proof}

\begin{lemma}\label{l:from_silting}
If the exceptional sequence $\Y$ is silting, then
the exceptional sequence $\mu_{\rev}(\Y)$ is a $\Hom_{\leq 0}$-configuration.
\end{lemma}

\begin{proof}
The proof is by induction. First consider the case $n= 2$. Let $(E,F)$ be an exceptional sequence,
and apply $\Hom(F,\ )$ to the approximation triangle 
$$E^{\ast} \to E \to F^r \to E^{\ast}[1]$$                                         
It follows that $(F,E^{\ast})$ is a $\Hom_{\leq 0}$-configuration.

Now, let $n>2$.
We use the presentation of $\mu_{\rev}$ defined
by (\ref{order}).  
After applying $\mu_1 \dots \mu_{n-1}$, we
obtain the exceptional sequence $(Y_n,Y_1^{\ast},\dots,Y_{n-1}^{\ast})$.  
Then $\Hom(Y_n,Y_i^{\ast}[j])=0$ for $i<n$, $j\leq 0$.  By Lemma \ref{mutationlemma}
(e), we know that $(Y_1^{\ast},\dots,Y_{n-1}^{\ast})$ is silting.  
By induction, applying $\mu_{\rev}^{(n-1)}$ to this silting exceptional
sequence
will yield a $\Homneg$-configuration. We know that the
mutations which are used are negative, that is to say, of the form

$$E_{i-1}[j]\rightarrow E_{i}^{\ast} \rightarrow E_i \rightarrow E_{i-1}[j+1]$$
with $j<0$. It follows that if we know that $\Hom(Y_n,E_i[t])$
and $\Hom(Y_n,E_{i-1}[t])$ vanish for $t \leq 0$, then also
 $\Hom(Y_n,E_i^{\ast}[t])$
and $\Hom(Y_n,E_{i-1}^{\ast}[t])$ vanish for $t \leq 0$. 
This shows that the mutations $\mu^{(n-1)}_{\rev}$ which we apply to reverse
$(Y_1^{\ast},\dots,Y_{n-1}^{\ast})$ preserve the property that
$\Ext^t(Y_n, \ )= 0$ for $t \leq 0$, and thus we are done.  
\end{proof}

\begin{lemma}\label{othermut}
Let $(Y_1,\dots,Y_n)$ be a $\Hom_{\leq 0}$-configuration exceptional sequence.  Consider the process 
of applying $\mu_{\rev}$ in the order given by (\ref{order}).  
Each mutation will be non-negative.
\end{lemma}

\begin{proof} The mutations which move $Y_n$ are all non-negative
since we begin with a $\Hom_{\leq 0}$-configuration.  The result holds by
induction, as in the proof of Lemma \ref{goodmut}.
\end{proof}

\begin{lemma}\label{l:from_config} 
If $\Y$ is a $\Hom_{\leq 0}$-configuration exceptional sequence, then 
the exceptional sequence $\mu_{\rev}(\Y)$ is 
silting.  \end{lemma}

\begin{proof} The proof is by induction, and the statement is easily verified in the case $n=2$.  
Assume $n>2$. We prove that $\mu_{\rev}(\Y)$ is silting using the order 
$(\ref{order})$. We apply $\mu_1 \mu_2 \dots \mu_{n-1}$ to
obtain the exceptional sequence $(Y_n, Y_1^{\ast}, \dots, Y_{n-1}^{\ast})$, and hence $\Ext^j(Y_n,Y_i^{\ast})= 0$  
for $j\geq 1$ and $1\leq i \leq n-1$.  
The sequence $(Y_1^{\ast}, \dots, Y_{n-1}^{\ast})$ is a $\Hom_{\leq 0}$-configuration,  by Lemma \ref{mutationlemma}, and hence applying
$\mu^{(n-1)}_{\rev}$ to this it will give a silting object by induction.

We then have to check that the mutations $\mu^{(n-1)}_{\rev}$ used in reversing the
$Y_i^{\ast}$'s preserve the property of $\Ext^j(Y_n,\ )$ vanishing for $j>1$.  

By Lemma \ref{othermut}, the approximations are of the form 
$$E_{i-1}[j]\rightarrow E_{i}^{\ast} \rightarrow E_i \rightarrow E_{i-1}[j+1]$$
with $j\geq 0$. The desired result is immediate.  
\end{proof}

\begin{proposition}\label{p:double_mu}
Let $\Y$ be an exceptional sequence. Then $\mu_{\rev}(\mu_{\rev}(\Y))=\nu^{-1}(\Y)$.
\end{proposition}

\begin{proof} 
We know that the effect of $(\mu_{n-1}\dots\mu_{1})$ is to
remove the left end term $Y_1$ from the exceptional sequence and replace it
with $\nu^{-1}(Y_1)$ at the right end. Thus, the effect of 
$(\mu_{n-1}\dots\mu_1)^n$ is to apply $\nu^{-1}$ 
to every element of the exceptional
sequence, maintaining the same order.  

Consider the operation
\begin{equation}\label{mess}
(\mu_{n-1}\dots\mu_{1})(\mu_{n-1}\dots \mu_{1})
\dots(\mu_{n-1}\dots\mu_{1})
\end{equation}
with $n$ repetitions of the product $(\mu_{n-1}\dots \mu_{1})$.  
This operation can be written as the composition of the following two operations
$$\mu' = (\mu_{n-1}\dots \mu_{1})(\mu_{n-1}\dots
\mu_{2})\dots (\mu_{n-1}\mu_{n-2})(\mu_{n-1})$$
and
$$\mu'' = (\mu_{1})(\mu_{2}\mu_{1})\dots
(\mu_{n-1}\dots \mu_{1}).$$
This can be done using only commutation relations, by taking the expression
(\ref{mess}) and 
moving to the left
the rightmost generator in the second parenthesis, the two rightmost in the third parenthesis,
etc. (counting from the left).  Both $\mu'$ and $\mu''$ are expressions for $\mu_{\rev}$.

\end{proof}

Summarizing, we have proved part (a) of Theorem \ref{maintheorem}, i.e. we have the following.

\begin{theorem}\label{parta}
The operation $\mu_{\rev}$ gives a bijection between 
silting exceptional sequences and $\Hom_{\leq 0}$-configuration exceptional sequences.
\end{theorem}
\begin{proof}
This is a direct consequence of Lemmas  \ref{l:from_silting} and \ref{l:from_config} and Proposition \ref{p:double_mu}, since
obviously $\nu$ gives a bijection on the set of all exceptional sequences.
\end{proof}

\section{The bijection between silting objects and $\Hom_{\leq 0}$-configurations}\label{s:bijection}

We have given a bijection from exceptional sequences coming from 
silting objects to exceptional sequences coming from $\Hom_{\leq 0}$-configurations.
We would like to show that this also determines a bijection from
silting objects to $\Hom_{\leq 0}$-configurations. This is not immediate from Theorem \ref{parta}, because
there can be more than one way to order a silting object or a 
$\Hom_{\leq 0}$-configuration into an exceptional sequence.  

We proceed as follows. Suppose we have a silting object $T$, and consider
some exceptional sequence $\E = (E_1, \dots, E_n)$ obtained from it.  
Consider the braid group $B_{n}=\langle \sigma_1,\dots,\sigma_{n-1}\rangle$,
where the action of $B_{n}$ on exceptional sequences is defined by having 
$\sigma_i$ act like $\mu_i$.  

Let $R_{\E}=\{(i,j)\mid \Ext^\bullet(E_i,E_j)=0\}$.  
Clearly, if we know $R_{\E}$, we know exactly which reorderings of $\E$ will be
exceptional sequences. 
Let $\stab_{\E} = \{\sigma \in B_{n} \mid \sigma \E = \E\} $ be the stabilizer.  

\begin{lemma}
$(i,j)\in R_{\E}$ if and only if $\mu^{-1}_{j-1} \dots \mu^{-1}_{i+1}(\mu_i)^2\mu_{i+1}\dots 
\mu_{j-1}\in \stab_{\E}$.  
\end{lemma}

\begin{proof} The effect of $\mu_{i+1} \dots \mu_{j-1}$ is to move $E_{j}$
to the left so it is adjacent on the right to $E_i$.  (This also modifies the 
elements it passes over.) We claim that
$\mu_i^2$ does not change $E$ if and only if $\Ext^\bullet(E_i,E_j)=0$.  In case
$\Ext^\bullet(E_i,E_j)=0$, the
remaining mutations $\mu^{-1}_{j-1}\dots 
\mu^{-1}_{i+1}$ undo the effect of the first mutations $\mu_{i+1}\dots \mu_{j-1}$, 
so the result is the identity. 

If $\Ext^\bullet(E_i,E_j) \neq 0$, then
the $i$-th element will be modified, and hence the composition $\mu^{-1}_{j-1} \dots \mu^{-1}_{i+1}(\mu_i)^2\mu_{i+1}\dots 
\mu_{j-1}$ is not in $\stab_{\E}$.   
\end{proof}

Denote by $\sigma_{\rev}$ the element of $B_{n}$ corresponding to $\mu_{\rev}$.  

From a basic lemma about group actions, we have that
$\stab_{\mu_{\rev}(E)}= \sigma_{\rev}\stab_E \sigma_{\rev}^{-1}$.  
To determine $\stab_{\mu_{\rev}}(\E)$, we need the following lemma.  (See
\cite{bri} for a different proof.)

\begin{lemma} In $B_{n}$, we have 
$\sigma_{\rev}\sigma_i\sigma_{\rev}^{-1}=\sigma_{n-i}$.
\end{lemma}

\begin{proof} Let $S_{n}$ be the symmetric group generated by the
  simple reflections $s_1, \dots, s_{n-1}$ and let $w_0$ be the longest element in 
$S_{n}$.  This is the 
permutation which takes $i$ to $n+1-i$ for all $i$.  For any $i$, we
can write $w_0= (w_0s_iw_0^{-1})(w_0s_i)$.  Note that
$w_0s_iw_0^{-1}= s_{n-i}$.  

For any $w\in S_{n}$, write $\sigma_w$ for the element of the braid group
$B_{n}$ obtained by taking any reduced word for $w$ and replacing each
occurrence of $s_i$ by $\sigma_i$ for all $i$.  This produces a
well-defined element of $B_{n}$ because any two reduced words for $w$ 
are related by braid relations, which also hold in $B_{n}$.  

Fix $i$, and write $u=w_0s_i$.  
We now have that $\sigma_{\rev}=\sigma_{w_0}=\sigma_{n-i}\sigma_u$.
So $\sigma_{\rev}\sigma_i\sigma_{\rev}^{-1}=\sigma_{n-i}\sigma_u\sigma_i\sigma_{\rev}^{-1}=\sigma_{n-i}\sigma_{\rev}\sigma_{\rev}^{-1}=\sigma_{n-i}$.  
\end{proof}

It follows that $(n-j, n-i) \in R_{\mu_{\rev}(\E)}$ if and only if $(i,j)\in R_{\E}$.
Hence we have proved the following, which is part (b) of our main theorem.

\begin{theorem}\label{partb}
The operation $\mu_{\rev}$ induces a bijection between
silting objects and $\Hom_{\leq 0}$-configurations.
\end{theorem}

\section{Specializing to $m$-cluster tilting objects and $m$-$\Hom_{\leq 0}$-configurations} \label{s:cluster_config}

In this section we prove part (c) of our main theorem.
We need to recall the following notions.
A full subcategory $\T$ of $\D$ is called {\em suspended} if
it satisfies the following:

\begin{itemize}
\item[(S1)] If $A \to B \to C \to A[1]$ is a triangle in $\D$ and $A,C$ are in $\T$, then $B$ is in $\T$.
\item[(S2)] If $A$ is in $\T$, then $A[1]$ is in $\T$.
\end{itemize} 

A suspended subcategory $\U$ is called a {\em torsion class} in \cite{br}
(or {\em aisle} in \cite{kv}) 
if the inclusion functor $\U \to \D$ has a right
adjoint. 
For a subcategory $\U$ of $\D$, we let ${\U}^{\perp} = \{X \in \D \mid \Hom(\U,X) = 0 \} $.
For a torsion class $\T$, let $\F = {\T}^{\perp}$ be the corresponding {\em torsion-free} class.
Recall that a torsion class in $\D$ is called {\em splitting} if
every indecomposable object in $\D$ is either torsion or torsion-free;
in other words, any indecomposable object which is not in the torsion class,
does not admit any morphisms from any object of the torsion class.  

We prove the following easy lemmas:

\begin{lemma} If $\E$ is an exceptional sequence contained in a splitting
torsion-free class $\pf$, then $\mu_{\rev}(\E)$ is also contained in 
$\pf$.
\end{lemma}

\begin{proof} This follows from the fact that each object in 
$\mu_\rev(\E)$ has a sequence of non-zero morphisms to an object in $\E$.
\end{proof}

\begin{lemma} If $\E$ is an exceptional sequence contained in a splitting 
torsion class $\pt$, then $\mu_{\rev}(\E)$ is contained in 
$\nu^{-1}(\pt)$.  
\end{lemma}

\begin{proof} This follows from the fact that, applying $\mu_{\rev}$ to
$\mu_{\rev}(\E)$, we obtain $\nu^{-1}(\E)$ by Proposition \ref{p:double_mu}, 
which implies that there is a
sequence of non-zero morphisms to
every element in
$\mu_{\rev}(\E)$ from an element in $\nu^{-1}(\E)$. 
\end{proof}

By combining the above lemmas, we obtain that $\mu_{\rev}$, 
applied to an exceptional sequence in 
$\dmoneplus$, yields a sequence with elements in $\dmzero$. Recall that an 
$m$-$\Hom_{\leq}$-configuration is a $\Hom_{\leq}$-configuration contained in $\dmzero$,
and that an $m$-cluster tilting object is a silting object contained in $\dmoneplus$. 
Hence, in particular we have the following.

\begin{proposition}\label{movingleft}
Let $\E$ be an exceptional sequence which is an $m$-cluster tilting object. 
Then $\mu_{\rev}(\E)$ is an $m$-$\Hom_{\leq 0}$-configuration.
\end{proposition}

We aim to show that the converse also holds.
For this we need the following lemmas.

\begin{lemma}\label{splittor} 
Let $\E$ be an exceptional sequence contained in a 
torsion class $\pt$.  Then if $\mu_i$ is non-negative,
$\mu_i(\E)$ is also contained in $\pt$.  
\end{lemma}

\begin{proof} When we apply a non-negative $\mu_i$, we have
the approximation sequence:
$$ E_{i-1}^r[j]\rightarrow E_i^{\ast} \rightarrow E_i \rightarrow E_{i-1}^r[j+1]$$
with $j\geq 0$.  The left and right terms of this triangle are in $\pt$, so
the middle term is also.
\end{proof}

\begin{corollary}\label{cor1} If a $\Hom_{\leq 0}$-configuration exceptional
sequence $\E$ is contained in 
a torsion class $\pt$, so is $\mu_{\rev}(\E)$.  
\end{corollary}

\begin{proof} By Lemma \ref{othermut}, we can calculate $\mu_{\rev}(\E)$
using only non-negative mutations.  The claim now follows from Lemma \ref{splittor}.
\end{proof}

\begin{lemma} \label{splittf} If $\E$ is an exceptional sequence contained in a 
torsion-free class $\pf$, and $\mu_i^{-1}$ is a negative
mutation, then $\mu_i^{-1}(\E)$ is also contained in $\pf$.  
\end{lemma}

\begin{proof} The proof is similar to that of Lemma \ref{splittor}.
The approximation is of the form
$$E_{i-1}\rightarrow E_{i-1}^{\ast} \rightarrow E_{i}^r[j]
\rightarrow E_{i-1}[1]$$
and $j\leq 0$.  Again, the left and right terms of the triangle are in $\pf$, 
hence the middle is also.
\end{proof}

\begin{corollary} \label{cor2} If $\E$ is a $\Hom_{\leq 0}$-configuration 
exceptional sequence which is contained in a torsion-free
class $\pf$, then $\mu_{\rev}^{-1}(\E)$ is contained in $\pf$.
\end{corollary}

\begin{proof} We know that $\mu_{\rev}^{-1}$ can be expressed as a 
product of negative mutations by Lemma \ref{goodmut}.
\end{proof}

\begin{proposition}\label{movingright} 
Let $\E$ be an
$m$-$\Hom_{\leq}$-configuration exceptional sequence. Then $\mu_{\rev}^{-1}(\E)$ is
an $m$-cluster tilting object.
\end{proposition}

\begin{proof} 
By Proposition \ref{p:double_mu} we have that
$\mu_{\rev}(\E)=\nu^{-1}\mu_{\rev}^{-1}(\E)$ and hence 
$\mu_{\rev}^{-1}(\E) = \nu \mu_{\rev}(\E)$. By Corollary \ref {cor1} we have that
$\mu_{\rev}(\E)$ is contained in $\dzero$. Hence $\mu_\rev^{-1}(\E)$ is contained in $\nu(\dzero)$. 
We also know
that $\mu_\rev^{-1}(\E)$ is contained in $\dm$ by Corollary \ref{cor2}.
This completes the proof.  
\end{proof}

Summarizing, we obtain part (c) of Theorem \ref{maintheorem}.

\begin{theorem}
The product of mutations $\mu_{\rev}$ defines a
bijection between $m$-cluster-tilting objects and $m$-$\Hom_{\leq 0}$-configurations.
\end{theorem}

\begin{proof}
This is a direct consequence of Propositions \ref{movingright} and \ref{movingleft}, using the
already established bijections from Theorems \ref{parta} and \ref{partb}.
\end{proof}

\section{A combinatorial interpretation: $m$-noncrossing partitions}\label{s:combi}

In this section, we give our desired combinatorial interpretation of part (c) of 
Theorem \ref{maintheorem}.  The main task of this section is to construct, for an arbitrary connected hereditary artin algebra $H$, 
a bijection between $m$-$\Homneg$-configurations and 
$m$-noncrossing partitions in the sense of \cite{arm} for the reflection
group $W$ corresponding to $H$.  

The set of $m$-clusters is only defined in 
the case that $H$ is of finite type; in this case, they are known to 
be in bijection with the $m$-cluster tilting objects \cite{t,z}.  Thus,
once we have accomplished the main task of this section, we will have 
obtained a bijection between $m$-clusters and $m$-noncrossing partitions
for $H$ of finite type (or equivalently, for $W$ any finite crystallographic
reflection group).  A description of the resulting bijection, in purely
Coxeter-theoretic terms, has already been
presented, without proof, in \cite{brt-fpsac}.

\subsection{Weyl groups and noncrossing partitions}
We define the Weyl group $W$ associated to $H$ following \cite{r}.  
Let $k$ be the centre of 
$H$ (which is a field since we have assumed that $H$ is connected).  
Number the simple objects of $H$ in such a way that 
$(S_1,\dots,S_n)$ is an exceptional sequence.  The Grothendieck group 
of $H$, denoted $K_0(H)$, is a free abelian group generated by the classes
$[S_i]$.  

For $i<j$, define 
\begin{align*}
\Delta_{ij}&=-\dim_{\End(S_i)}(\Ext^1(S_i,S_j))\\
\Delta_{ji}&=-\dim_{\End(S_j)}(\Ext^1(S_i,S_j))
\end{align*}
Write $d_i$ for the $k$-dimension of $\End(S_i)$.  Note that 
$d_i\Delta_{ij}=d_j\Delta_{ji}$.  
Now define a 
symmetric, bilinear form on $K_0(H)$ by $([S_i],[S_j])=d_i\Delta_{ij}$ for
$i\ne j$, and $([S_i],[S_i])=2d_i$.

For $x$ in $K_0(H)$, with $(x,x)\ne 0$, 
define $t_x$, the reflection along $x$, by:

$$t_x(v)=v - \frac{2(v,x)}{(x,x)} x$$

We now have the following lemma:
\begin{lemma} \label{formlemma}
\begin{enumerate}\item[(a)] For $A,B$ modules, we have  
$([A],[B])= \dim_k\Hom(A,B)+\dim_k\Hom(B,A)-\dim_k\Ext^1(A,B) -
\dim_k\Ext^1(B,A).$ 
\item[(b)]
If $(A,B)$ is an exceptional sequence in $\D$, 
and $(B,A^*)$ is the result of mutating it, then:
$$[A^*]= t_{[B]}[A].$$
\end{enumerate}
\end{lemma}

\begin{proof} (a) See \cite[p. 279]{rin1}.

(b) Let $(A,B)$ be an exceptional sequence in $\D$.  Consider the triangle
$A^*\rightarrow A \rightarrow B^r[p]$, where $f:A\rightarrow B^r[p]$ is a 
minimal left $\th(B)$-approximation.  Note that we know that 
$\Hom(A,B[j])=0$ for all but at most one $j$.  Without loss of generality,
assume $p=0$.

Assume that $f:A\rightarrow B^r$ is non-zero.  We have that 
$K_0(H)\isom K_0(\D)$, and the above triangle gives $[A^*]=[A]-
r[B]$.  Since $B$ is exceptional, and hence $\End(B)$ is a division
ring, it follows directly that $f=(f_1,\dots,f_r)$ where 
$\{f_1,\dots,f_r\}$ is an $\End(B)$-basis for $\Hom(A,B)$.  Hence, 
$r=\dim_{\End(B)}\Hom(A,B)$.  Since 
$$t_{[B]}[A] = [A]-\frac{2([A],[B])}{([B],[B])} [B],$$
it suffices to show that $r= 2([A],[B])/([B],[B])$.  We have that
$([A],[B])=\dim_k\Hom(A,B)$ by extending (a) to $\D$.  (We use here
that $\Ext^i(B,A)=0$ since $(A,B)$ is an exceptional sequence,
and $\Ext^1(A,B)=0$ since $\Hom(A,B)\ne 0$.) 
By (a), we similarly get that $([B],[B])=
2\dim_k(\Hom(B,B))$.  Hence we have:
$$\frac{2([A],[B])}{([B],[B])} = \frac{2\dim_k\Hom(A,B)}{2\dim_k\Hom(B,B)}
= \dim_{\End(B)}\Hom(A,B)=r,$$
and we are done in this case.

If $f=0$, so that $r=0$, then $[A^*]=[A]$, and $([A],[B])=0$, so 
$t_{[B]}([A])=[A]$, as desired.  
\end{proof}

We now define $s_i=t_{[S_i]}$ for $1\leq i \leq n$, and let $W$ be the group
 generated by the set $\{s_1,\dots,s_n\}$; it is a Weyl group.  
By definition, we say that
an element of $W$ is a {\em reflection} if it is the conjugate of some $s_i$. 
We denote the set of all reflections in $W$ by $T$.  

It follows directly from Lemma \ref{formlemma} that if $(E_i,E_{i+1})$ and $(E_{i+1},E_i^*)$ are related
by mutation, then 
\begin{equation}\label{lengthtwo}
t_{[E_i]}t_{[E_{i+1}]}=t_{[E_{i+1}]}t_{[E_i^*]}.\end{equation}  
Since the 
reflections corresponding to the simple objects are all in $W$, it follows
from (\ref{lengthtwo}) that $t_{[E]}$ is in $W$ for each exceptional module $E$.
 It also follows
from (\ref{lengthtwo}) that the product of the reflections corresponding to any
exceptional sequence is the Coxeter element $c$ (see \cite{it}).  

Conversely, we have the following result from \cite{is}:

\begin{theorem}\label{is} If $c=t_1\dots t_n$ in $W$, with $t_i\in T$,
then each $t_i$ must be of the form $t_{[E_i]}$ for some 
exceptional module $E_i$, where $(E_1,\dots,E_n)$ forms an
exceptional sequence.   
\end{theorem}

Now we give the (purely Coxeter-theoretic) 
definition of an $m$-noncrossing partition.  
First of all,
define a function $\ell_T:W\rightarrow \mathbb N$, where $\ell_T(w)$ is
the length of the shortest expression for $w$ as a product of reflections.
(Note that this is not the classical length function on $W$, which is
the minimum length of an expression for $w$ as a product of 
{\it simple} reflections.)  We note that $\ell_T(c)=n$.  

We say that $(u_1,\dots,u_r)$, an $r$-tuple of elements of $W$, 
is a {\em $T$-reduced expression} for $u_1\dots u_r$
if 
$\ell_T(u_1)+\dots+\ell_T(u_r)=\ell_T(u_1\dots u_r)$.  We can now follow
Armstrong \cite{arm} in defining
the {\em$m$-noncrossing partitions} for $W$ to consist of the set of 
$T$-reduced expressions for $c$ with $m+1$ terms.  

Now we define the bijection.  Let $(u_1,\dots ,u_{m+1})$ be a $T$-reduced
expression for $c$.  By Theorem \ref{is}, pick an exceptional sequence
$E_1,\dots,E_n$ such that the first $\ell_T(u_1)$ terms correspond to
some factorization of $u_1$ into reflections, and similarly for the 
next $\ell_T(u_2)$ terms, and so on.  
For each $i$ with $1\leq i \leq m+1$, we then
have an exceptional sequence $\mathcal E_i$.  Write 
$\mathcal C_i$ for the minimal abelian subcategory containing 
$\mathcal E_i$.  Let $F_i$ be the sum of the simples of $\mathcal C_i$.  
Set $\phi(u_1,\dots,u_{m+1})=\bigoplus F_i[m+1-i]$.  

\begin{theorem} \label{conftonc} The above map $\phi$ from $T$-reduced expressions of $c$ to
objects in $\D$ is a bijection from $m$-noncrossing partitions to
$m$-$\Homneg$-configurations.  \end{theorem}

\begin{proof} First, we show that if $(u_1,\dots,u_{m+1})$ is a 
$T$-reduced expression for $c$, then $\bigoplus F_i[m+1-i]=\phi(u_1,\dots,u_{m+1})$
is an $m$-$\Homneg$-configuration.  
By definition, $\bigoplus F_i[m+1-i]$ is contained in
$\mathcal D_{\geq 0}^{\leq m}$.  We check the four conditions in the 
definition of a $\Homneg$-configuration.  (H1) is immediate.  
It is possible
to transform each sequence $\mathcal E_i$ into (an ordering of) 
the summands of $F_i$ by mutations, thanks to the transitivity of the action
of mutations within $\mathcal C_i$.  (H4) follows, and the form of
this exceptional sequence guarantees (H2) and (H3).

Next, we show that any $m$-$\Homneg$-configuration arises in this way.  
Take $X$ to be an $m$-$\Homneg$-configuration and order it into an
exceptional sequence in such a way that the objects in 
$\mod H[m]$ come first, then those in $\mod H[m-1]$, etc.  This was shown
to be possible in the proof of Lemma \ref{hfour}.  

Now, for $1\leq i \leq m+1$, define $\mathcal C_i$ to be the 
subcategory of $\mod H$ consisting of modules admitting a filtration
by modules corresponding to the summands in $X$ of degree $m+1-i$.  This is
the minimal abelian subcategory of $\mod H$ containing these summands
of $X$, and the summands of $X$ are obviously the simple objects in this
subcategory.  We can therefore define $u_i$ by taking the product of
these summands of $X$, ordered as in the exceptional sequence, and we
obtain a $T$-reduced expression for $c$.  
\end{proof}

\subsection{Combinatorics of positive Fuss-Catalan numbers}
When $H$ is of finite type, corresponding to a finite 
crystallographic group $W$, there is a variant of the Fuss-Catalan
number called the {\it positive Fuss-Catalan number}, denoted
$C^+_m(W)$.  By definition, $C^+_m(W)=|C_{-m-1}(W)|$.  

It is known that the number of 
$m$-cluster tilting objects contained in $\dmone$ is  
$C^+_m(W)$, see \cite{fr}.

Write
$\dmzerominus$ for the full subcategory of
$\dmzero$ additively generated by the indecomposable
objects of $\dmzero$ other than the summands of $H$.  
The following is an immediate corollary of Theorem \ref{maintheorem}.

\begin{corollary} \label{positivebij}
\begin{enumerate}\item[(a)]There is a bijection between silting objects contained
in $\dmone$ and $m$-$\Homneg$-configurations contained in 
$\dmzerominus$ given by 
$\mu_{\rev}$.  
\item[(b)] If $H$ is of Dynkin type with corresponding
crystallographic reflection group $W$, then the number of 
$m$-$\Homneg$-configurations contained in $\dmzerominus$
is $C^+_m(W)$.  
\end{enumerate} 
\end{corollary}

It is possible to give a Coxeter-theoretic description of
the subset of $m$-noncrossing partitions which correspond, 
under the bijection of Theorem \ref{conftonc}, 
to
the $m$-$\Homneg$-configurations contained in $\dmzerominus$.  
See \cite{brt-fpsac} for more details.

\section{The link between $\Homneg$-configurations and Riedtmann's
combinatorial configurations}

In this section we show that our $\Homneg$-configurations contained in
$\donezerominus$ are related to the combinatorial configurations 
introduced by Riedtmann in connection with her work on selfinjective
algebras of finite representation type. Note that an alternative and
independent approach to this, dealing with the Dynkin case, is given by Simoes \cite{sim}.  She also gives a bijection from combinatorial configurations
to a subset of the $1$-noncrossing partitions, and thus to the positive 
clusters (in the sense of the previous section).

\subsection{Complements of tilting modules and cluster-tilting objects}
In this subsection we recall some basic facts about complements
of tilting modules in $\mod H$ and cluster tilting objects in the
associated cluster category.  For more on complements of tilting modules,
see \cite{hu,rs1,u,chu}; for more on complements in cluster categories, see 
\cite{bmrrt}.  

Suppose that $T=\bigoplus_{i=1}^n T_i$ is a tilting object in $\mod H$.
Write $\overline T$ for $\bigoplus _{j\ne i} T_j$.   

We say that an indecomposable object $X$ in  $\mod H$ is a complement to $\overline T$ if $X\oplus \overline T$
is tilting.  If $\overline T$ is not sincere, then $T_i$ is its only
complement; otherwise, it has exactly two complements up to isomorphism, 
$T_i$ and one other
one, $T_i'$.  We say that $T_i'\oplus \overline T$ is the result of 
mutating $T$ at $T_i$.  

\begin{lemma} \label{trich}
Exactly one of the following three possibilities occurs:
\begin{enumerate}
\item [(a)]{\bf $T_i$ has no replacement.}  This occurs if and only if
$\overline T$ is not sincere.  
\item [(b)]{\bf $T_i$ admits a monomorphism to a module in $\add \overline T$.}
In this case, let $T_i \rightarrow B$ be the minimal left 
$\add \overline T$-approximation
to $T_i$.  Then there is a short exact sequence:
$$0 \rightarrow T_i\rightarrow B \rightarrow T_i' \rightarrow 0.$$
\item [(c)]{\bf $T_i$ admits a epimorphism from a module in 
$\add \overline T$.}  In this case, let $B\rightarrow T_i$ be the
minimal right $\add \overline T$-approximation to $T_i$.  Then 
there is a short exact sequence 
$$0 \rightarrow T_i' \rightarrow B \rightarrow T_i \rightarrow 0.$$
\end{enumerate}
\end{lemma}

We also think of $\mod H$ as embedded
inside the cluster category associated to $H$.  
A tilting object in $\mod H$ is thereby identified with a (1-)cluster tilting object in the 
cluster category.  
In the cluster category, there is 
always exactly one way to replace $T_i$ by some other indecomposable
object while preserving
the property of being a cluster tilting object.  If there is a replacement
for $T_i$ in $\mod H$, that replacement is also a replacement in 
the cluster category; otherwise, the replacement for $T_i$ is of the form
$P[1]$, where $P$ is indecomposable projective.  

\subsection{Torsion classes arising from partitions of exceptional sequences}  This subsection is mainly devoted to the proof of
Lemma \ref{ttf}, which says that if a complete exceptional sequence
in $\mod H$ is divided into two parts, $(E_1,\dots,E_r)$ and 
$(E_{r+1},\dots,E_n)$, for some $0<r<n$, 
and the objects from the second part are
used to generate a torsion class, then the corresponding torsion-free
class is generated (in a suitable sense) by the objects from the first part 
of the exceptional sequence.  

Let $T$ be a tilting module, $\mathcal T=\Fac T$ the torsion class
generated by $T$, and 
$\mathcal F=\Sub \tau T$ the corresponding torsion-free class.

Some summand $U$ of $T$ (typically not indecomposable) is minimal among
modules such that $\Fac U=\mathcal T$.  We refer to $U$ as the minimal
generator of $\mathcal T$.  Similarly, there is a minimal cogenerator of 
$\mathcal F$.  

We have the following lemma, based on an idea from \cite{it}.

\begin{lemma}\label{altgen} Let $T_i$ be an indecomposable summand of $T$.  
Then $T_i$ is a summand of the minimal generator of $\mathcal T$ if and only if
$\tau T_i$ is not a summand of the minimal cogenerator of 
$\mathcal F$.  (By convention, if $\tau T_i=0$, then
we do not consider it a summand of the minimal cogenerator of $\mathcal F$.)
\end{lemma}

\begin{proof} 
If $T_i$ is projective, then it must be a summand of the 
minimal generator for $\mathcal T$,
and then $\tau T_i$ is zero, so (by convention) 
it is not a summand of the minimal generator
for $\mathcal F$.  
We may therefore assume that $T_i$ is not projective.  

For the rest of the proof, we embed $\mod H$ into the corresponding
cluster category.  Note that $\tau$ is an autoequivalence on the 
cluster category.

Let $T'_i$ be the result of mutating $T$ at $T_i$ in
the cluster category.  
Since $\tau$ is an autoequivalence, the effect of mutating $\tau T$ at $\tau T_i$ is to replace $\tau T_i$ by $\tau T_i'$.  Write $\overline T$ for 
$\bigoplus_{j\ne i} T_j$.  

Suppose now that $T_i$ is a summand of the minimal generator
for $\mathcal T$.  
Then
there is no epimorphism from $\add \overline T$ to $T_i$,
so either 
there is a short exact
sequence in the module category
$$0\rightarrow T_i \rightarrow B \rightarrow T_i'\rightarrow 0,$$
where $B$ is in $\add \overline T$, or else
$T_i'$ is a shifted projective.  

In the former case, 
applying $\tau$ to the above sequence shows that $\tau T_i$ is not a
summand of the minimal cogenerator of $\mathcal F$, since 
$\tau T_i$ injects into $\tau B\in \add \tau \overline T$.  
In the latter case, $\tau T_i'$ is injective, so the exchange sequence
in $\mod H$ again has the same form ($\tau T_i$ is on the left, and 
therefore injects into an object of $\add \tau \overline T$, so is
not a summand of the minimal cogenerator of $\mathcal F$).  

Next suppose that $T_i$ is not a summand of the minimal generator for 
$\mathcal T$.  So there is an epimorphism from some $B $ in $\add \overline T$
to $T_i$, and thus we have a short exact sequence in $\mod H$ of
the form
$$0\rightarrow T_i' \rightarrow B \rightarrow T_i \rightarrow 0.$$
Therefore either $\tau$ applied to the above sequence in $\mod H$ is still a
short exact sequence, or else $T_i'$ is projective, and hence 
$\tau T_i'$ is a shifted projective.  In the first case, the exchange sequence
for $\tau T_i$ has $\tau T_i$ on the right; in particular, $\tau T_i$ does
not admit a monomorphism to any $B'$ in  $\add \tau \overline T$.  Thus
$\tau T_i$ is a summand of the minimal cogenerator of $\mathcal F$.  In the
second case, $\tau T_i$ has no complement in $\mod H$, so $\tau \overline T$
is not sincere and thus $\tau T_i$ is again a summand of the minimal cogenerator
for $\mathcal F$.  
\end{proof}

A subcategory of $\mod H$ is called \emph{exact abelian} if it is abelian
with respect to the exact structure inherited from $\mod H$.  
If $(E_1,\dots,E_r)$ is an exceptional sequence in $\mod H$, it naturally
determines an exact abelian and extension-closed subcategory of $\mod
H$, the smallest such subcategory of $\mod H$ containing $E_1,\dots,E_r$.  
This
subcategory is a module category for a hereditary algebra $H'$ with
$r$ simples \cite{r}.  If $(E_1,\dots,E_n)$ is a complete exceptional 
sequence, then the minimal exact abelian and extension-closed subcategory 
of $\mod H$ containing $E_1,\dots,E_r$ can also be described as the
full subcategory of $\mod H$ consisting of all $Z$ such that 
$\Hom(E_i,Z)=0=\Ext^1(E_i,Z)=0$ for all $r+1\leq i \leq n$.  

\begin{lemma}\label{ttf}
Let $(E_1,\dots,E_n)$ be a complete exceptional sequence in $\mod H$.  
Let $\mathcal B$ be the exact abelian extension-closed 
subcategory generated by 
$E_1,\dots,E_r$, with $0< r <n$, 
and let $\mathcal C$ be the exact abelian 
extension-closed subcategory 
generated by $E_{r+1},\dots,E_{n}$.  Let $\mathcal T=\Fac \mathcal C$, and
$\mathcal G=\Sub \mathcal B$.  Then $(\mathcal T,\mathcal G)$ forms a 
torsion pair.  
\end{lemma}

\begin{proof} 
Since $\mathcal C$ is closed under extensions, it is straightforward to
see that $\mathcal T$ is also closed under extensions, and hence that it
is a torsion class.    
Let $\mathcal F$ be the 
torsion-free class corresponding to $\mathcal T$.  Clearly
$\mathcal G$ is a full subcategory of $\mathcal F$.  
Suppose first that $\mathcal T$
is generated by a tilting object $T=\bigoplus T_i$, so we can apply
Lemma \ref{altgen}.  Let $P$ be the minimal generator of $\mathcal T$.
This consists of the direct sum of the indecomposable 
$\Ext$-projectives of $\mathcal C$. (Note that 
$\mathcal C$ is again a module category.) 
Let $T_i$ be a summand of $T$ which is not
a summand of the minimal generator of $\mathcal T$.  
Since $\tau T_i$ is in $\mathcal F$, we know
that $\Hom(P,\tau T_i)=0$.  Let $P_j$ be an indecomposable 
summand of $P$.  We want to show
that $\Hom(T_i,P_j)=0$.  
Morphisms between indecomposable summands of a tilting object
are epimorphisms or monomorphisms \cite{hr}.  
Since $P_j$ is by assumption a summand of the
minimal generator of $\mathcal T$, it cannot admit an epimorphism from $T_i$.
Since $T_i$ admits an epimorphism from $\overline T$, it cannot also admit
a monomorphism into $P_j$ (by Lemma \ref{trich}).  
Therefore, $\Hom(T_i,P)=0$, and hence
$\Ext^1(P,\tau T_i)\isom D\underline{\Hom}(T_i,P)=0$.
Using the remarks before the statement of the lemma,
  we conclude that $\tau T_i$ lies in $\mathcal B$.  By Lemma \ref{altgen} we conclude
that all the indecomposable summands of the minimal cogenerator of 
$\mathcal F$ lie in $\mathcal B$, and therefore in $\mathcal G$.   
So
$\mathcal F=\mathcal G$, as desired.  

Suppose now that $\mathcal T$ is not generated by a tilting module.  
It is still generated by the direct sum of the indecomposable 
non-isomorphic $\Ext$-projectives of 
$\mathcal C$, which we denote by $T$.  
Let $I_1,\dots,I_s$ be the indecomposable injectives such that 
$\Hom(T,I_i)=0$.  These are objects of $\mathcal B$.  Suitably
ordered, $(I_1,\dots,I_s)$ form an exceptional sequence in $\mathcal 
B$; we can therefore extend this sequence to a complete exceptional sequence
in $\mathcal B$, which we denote by
$(I_1,\dots,I_s,F_1,\dots,F_{r-s})$ .  Note that this sequence can
be further extended to a complete 
exceptional sequence in $\mod H$ by appending
$(E_{r+1},\dots,E_n)$.  

Consider the category $\mathcal M$ with objects
$\{M\mid \Hom (M,I_i)=0$ for $1\leq i \leq s\}$.  This is a module category
for some hereditary algebra $H'$ with $n-s$ simples.  $\mathcal T$ is
a torsion class for $\mod H'$, and $T$ is tilting in $\mod H'$.  We can
therefore apply the previous case to conclude that 
the torsion-free class in $\mod H'$ associated to $\mathcal T$ is 
cogenerated by $\mathcal C'$, the smallest exact abelian extension-closed 
subcategory of $\mod H'$
containing $F_1,\dots,F_{r-s}$.   Now if $Z$ is any object in 
$\mod H$, we want to show that there is an exact sequence
$$ 0 \rightarrow K \rightarrow Z \rightarrow Z/K\rightarrow 0$$
with $Z/K$ in $\mathcal G$ and $K$ in $\mathcal T$.  If we can do this,
then that shows that $\mathcal G$ is ``big enough'', that is to say,
it coincides with $\mathcal F$.   

To do this, let $N$ be the maximal quotient of $Z$ which is a subobject
of $\add \oplus_{i=1}^s  I_i$, and let the kernel be $Z'$.  So $Z'$ admits
no non-zero morphisms to $\oplus_{i=1}^s I_i$; in other words, $Z'$ is in 
$\mod H'$.  So $Z'$ has a maximal torsion submodule $K$, and $Z'/K$ is
in the torsion-free class associated to $\mathcal T$ in $\mod H'$.  
It follows that $Z/K$ is in $\mathcal G$, and we are done.
\end{proof}

\subsection{Riedtmann's combinatorial configurations}

Define the autoequivalence $F=[-2]\tau^{-1}$ of $\D$.

A collection $\mathcal I$ 
of indecomposable objects in $\D$ is called a (Riedtmann)
combinatorial configuration if it satisfies the following two
properties:
\begin{itemize}
\item For $X$ and $Y$ non-isomorphic objects in $\mathcal I$, we have 
$\Hom(X,Y)=0$.  
\item For any nonzero $Z$ in $\D$, there is some $X\in \mathcal I$ such that
$\Hom(X,Z)\ne 0$.  
\end{itemize}
Note that Riedtmann only considers combinatorial configurations for path
algebras of type ADE, but the above definition does not require that
restriction.  

A combinatorial configuration is called {\it periodic} if it satisfies
the additional property that (in our notation) for any $X\in \mathcal I$,
we have $F^i(X)\in \mathcal I$ for all $i$.  
Riedtmann showed that if $H$ is a
path algebra of type $A$ or $D$, then any combinatorial configuration is 
periodic \cite{Rie1,Rie2}.  

\begin{theorem}\label{eightfour} If $T$ is a $\Homneg$-configuration contained in 
$\donezerominus$, then the set of indecomposable summands of $F^i(T)$ for all $i$ 
is a periodic combinatorial configuration
in the sense of Riedtmann. \end{theorem}

\begin{proof}
To verify the $\Hom$-vanishing condition in the definition of 
a combinatorial configuration, it suffices to verify, for 
any non-isomorphic indecomposable summands $A,B$ of $T$, that 
$\Hom(A,F^i(B))=0$.  It is clear that $\Hom(A,F^i(B))$ is zero unless $i=0$ or
$i=-1$.  If $i=0$, the vanishing follows directly from the definition
of a $\Homneg$-configuration.  For $i=-1$, observe that
$\Hom(A,F^{-1} B)\isom D \Ext^{-1}(B,A)=0$.  

Let $\hat T=\bigoplus_{i} F^i(T)$.  
Now we consider the property that for each $X$ in $\D$, we have that 
$\Hom(\hat T,X)\ne 0$.  
We may assume that $X$ is indecomposable.
We can clearly 
assume that $X\in \donezerominus$.   
%We recall the notation that if $X$ is an indecomposable object in
%$\D$, we write $\overline X$ for the corresponding module

%, and write $\overline X$ for the corresponding 
%module.  
Let $E_1[1],\dots,E_r[1]$ be the indecomposable summands of $T$ in degree 1, and let
$E_{r+1},\dots,E_n$ be the indecomposable summands of $T$ in degree 0, ordered so
that $(E_1,\dots,E_n)$ forms an exceptional sequence in $\mod H$. 

Assume first
that $X$ is in degree 0.
Let $\mathcal B$ be the smallest exact abelian extension-closed
subcategory containing 
$E_1,\dots,E_r$.  This is the category of objects of $\mod H$ filtered by 
$\{E_1,\dots,E_r\}$. 
Similarly, let $\mathcal C$ be the smallest exact abelian extension-closed
subcategory containing 
$E_{r+1},\dots,E_n$.  

Let $\mathcal T=\Fac \mathcal C$, and $\mathcal F=\Sub \mathcal B$.  By
Lemma \ref{ttf}, $(\mathcal T,\mathcal F)$ is a torsion pair.  

If $X$ has non-zero torsion, then we have shown that $X$ admits a 
non-zero morphism from some object in $\mathcal T$, and therefore from some
object of $\mathcal C$, so $X$ admits a non-zero morphism
from some 
$E_i$ with $r+1\leq i \leq n$.  Since this $E_i$ is a summand of $\hat T$, we
are done with this case.  

Now suppose that $X$ has no torsion, which is to say, 
it is torsion-free.  
$X$ therefore admits a monomorphism into some object of 
$\mathcal B$, 
and 
thus a non-zero morphism to some $E_i$ with $1\leq i \leq r$.  Hence there is 
a non-zero morphism
from $\nu^{-1}(E_i)$ to $X$.  
But $\nu^{-1}(E_i)=F(E_i[1])$, which is a summand of 
$\hat T$, and we are done.  

Now consider the case that $X$ lies in degree 1.
Let $Z=(\bigoplus_{i=1}^r E_i[1]) \oplus 
(\bigoplus_{i=r+1}^n F^{-1}E_i)$.  We claim that $Z$ is a $\Homneg$-configuration contained in $\D_{(\leq 2)-}^{(\geq 1)}$ (by which we mean
$\D_{(\leq 2)}^{(\geq 1)}$ with $DH[2]$ removed).  
(H1) is clear.  For (H2), the nontrivial requirement is to show 
$\Hom(E_i[1],F^{-1}E_j)=0$
with $i\leq r$ and $j>r$.  Now $\Hom(E_i[1],F^{-1}E_j)\simeq 
D\Ext^{-1}(E_j,E_i[1])=0$.  For (H3), the nontrivial requirement is to
show that $\Ext^{-1}(E_i[1],F^{-1}E_j)=0$ for $i\leq r$ and $j>r$, 
and we see that 
$\Ext^{-1}(E_i[1],F^{-1}E_j)\simeq \Hom(E_j,E_i[1])=0$.  
For (H4), observe that, by Lemma \ref{mutationlemma}(c), $(\mu_{1}\dots\mu_{n-1})^{n-r}$ transforms 
$(E_1[1],\dots,E_r[1],E_{r+1},\dots,E_n)$ to $(\nu E_{r+1},\dots,
\nu E_n,E_1[1],\dots,E_r[1])$.  Up to some shifts of degrees, the terms
in this exceptional sequence coincide with the summands of $Z$, which
implies (H4).  

Now apply the argument from the case that $X$ is in degree zero to $X[-1]$ and 
the $\Homneg$-configuration $Z[-1]$.  
\end{proof}

\begin{theorem}\label{eightfive} If $\mathcal I$ is a periodic combinatorial configuration, 
and $H$ is of finite type, then the objects of $\mathcal I$ lying
inside $\donezerominus$ form a $\Homneg$-configuration.  \end{theorem}

\begin{proof} Let $X$ be the direct sum of the objects of $\mathcal I$
lying inside $\donezerominus$.
We show first of all that $X$ has at least $n$ non-isomorphic
indecomposable summands.  
By the definition of combinatorial configuration, any object in
$\mod H[1]$ admits a non-zero morphism from some object in 
$\mathcal I$. By degree considerations, such an object must
be a summand of $X$.  Thus $X$ generates $\D$, and therefore contains
at least $n$ non-isomorphic indecomposable summands. 

It follows from the definition of combinatorial configuration that 
$\Hom(X,X)$ has as basis the identity maps on the 
indecomposable summands of $X$.  If $A,B$ are two non-isomorphic
indecomposable summands of $X$, we have that $\Ext^{-1}(A,B)\isom
D\Hom(F(B),A)$ is zero, since $F(B)\in \mathcal I$.  Further, $\Ext^t(A,B)=0$
for $t<-1$ because $X$ is contained in $\donezerominus$.  
Since $H$ is of
Dynkin type, the summands of $X$ are exceptional and also
(H4) holds.  It follows that the summands of 
$X$ can be ordered into an exceptional sequence, which means that there
are at most $n$ of them, so there are exactly $n$, and $X$ is a
$\Homneg$-configuration.  
\end{proof}

Note that silting objects in $\D^{\geq 1}_{\leq 1}$ naturally correspond
to tilting $H$-modules.  
Combining Theorems \ref{eightfour} and \ref{eightfive} with Corollary 
\ref{positivebij}, we obtain the following corollary.

\begin{corollary} Assume that the hereditary algebra $H$ is of Dynkin type. 
Then there is a natural bijection between the tilting $H$-modules and
the periodic combinatorial configurations. \end{corollary}

A bijection between the tilting $H$-modules and the periodic 
combinatorial configurations was constructed in type ADE in \cite{BLR}.

\section{Torsion classes in the derived category}

Both silting objects and torsion classes play an important role in this paper. 
Here we point out that there is a close relationship between these
concepts.  

For an object $M$ in $\D$ we can define (as in \cite{kv})
the subcategory 
$$A(M)=\{X \in \D \mid \Ext^i(M,X)=0 \textrm { for } i \geq 1 \}.$$  
In this section we prove that $A(M)$
is preserved under application of $\mu_{\rev}$.

\begin{lemma} If $M$ is silting, $A(M)$ is a torsion class.  
\end{lemma}

\begin{proof}
By \cite[Cor. 3.2]{ast} (see \cite{kv} in the Dynkin case) the smallest suspended subcategory $U(M)$ containing $M$ is a torsion class.
We claim that $A(M) = U(M)$. Since $A(M)$ is clearly suspended, we need only to show $A(M) \subset U(M)$.
Assume $X$ is in $A(M)$. Since $U(M)$ is a torsion class, there is (see \cite{ast, br})  
a triangle 
$$U \to X \to Z \to U[1]$$ 
in $\D$ with $U$ in $U(M)$ and with $Z$ in ${U(M)}^{\perp}$. Since $U[1]$ is in $U(M) \subseteq A(M)$, and $A(M)$ is suspended,
we also have that $Z$ is in $A(M)$. 

By Lemma \ref{silting_generates} we have that $M$ is a generator. Since $Z$ is in ${U(M)}^{\perp}$ and $M[i]$ is in $U(M)$ for 
$i \geq 0$, we have that $\Hom_{\D}(M, Z[i]) = 0$ for $i \leq 0$. On the other hand, since $Z$ is in $A(M)$ we have
by definition that $\Hom_{\D}(M, Z[i]) = 0$ for $i>0$. Hence $Z= 0$, and $X \simeq U$ is in $U(M)$.
\end{proof}

The following can be found in \cite{ast}.

\begin{proposition} If $A$ is a torsion class which is $A(Y)$ for some silting object $Y$, then
$Y$ can be recovered as the $\Ext$-projectives of $A$.  
\end{proposition}

From this we obtain the following direct consequence.

\begin{corollary} The map $Y \mapsto A(Y)$ is an injection from silting objects to
torsion classes. 
\end{corollary}

The following shows that the torsion class associated to an exceptional 
sequence is not affected by negative mutations.  

\begin{proposition}\label{preserve}
If $\mu_i$ is a negative mutation for $Y$, then
$A(Y)=A(\mu_i(Y))$.
\end{proposition}

\begin{proof} 
Consider an approximation triangle
$$Y_{i+1}^r[j] \rightarrow Y_i^{\ast} \rightarrow Y_i \rightarrow Y_{i+1}^r[j+1]
$$
with $j$ negative. 
The result follows from the long exact sequence obtained by applying $\Hom(\ ,X)$ to this triangle.
\end{proof}

Hence the correspondence described in Section \ref{s:silt_hom} preserves the torsion classes.

\begin{corollary} If $Y=(Y_1,\dots,Y_n)$ is silting, then 
$A(\mu_{\rev}(Y))=A(Y)$.  
\end{corollary}

\section*{Acknowledgements}

The third author would like to thank Drew Armstrong, Chris Brav, and
David Speyer
for helpful conversations.  We also thank the referee for his helpful
suggestions.  Much of the work on this paper was done during several
visits by the third author to NTNU.  He would like to thank his co-authors
and the members of the Institutt for matematiske fag for their
hospitality.

\end{document}